\documentclass{article}

\usepackage{mathtools}
\usepackage{amssymb,amsfonts,amsmath,amsthm}
\usepackage{cite}
\usepackage[hidelinks]{hyperref}
\usepackage{indentfirst}

\usepackage{color}
\usepackage{graphicx}
\usepackage{geometry}
\usepackage{longtable}
\usepackage{subcaption} 
\usepackage{caption}
\usepackage{multirow}
\usepackage{array}

\usepackage{fancyhdr}
\usepackage{booktabs}
\newtheorem{theorem}{Theorem}[section]
\newtheorem{lemma}[theorem]{Lemma}
\newtheorem{corollary}[theorem]{Corollary}

\newtheoremstyle{questionstyle}
  {\topsep} 
  {\topsep} 
  {\itshape} 
  {} 
  {\bfseries} 
  {.} 
  {.5em} 
  {\thmname{#1}\thmnumber{ #2}\thmnote{ (#3)}} 

\theoremstyle{questionstyle}

\newtheoremstyle{conjecturestyle}
  {\topsep} 
  {\topsep} 
  {\itshape} 
  {} 
  {\bfseries} 
  {.} 
  {.5em} 
  {\thmname{#1}\thmnumber{ #2}\thmnote{ (#3)}} 

\theoremstyle{conjecturestyle}
\newtheorem{conjecture}{Conjecture}[section] 

\title{On the local average order of dominating sets{\footnote {Supported by National Natural Science Foundation of China (Grant
Number: 12426663)}}}

\author{Tingyun Chen$^{a}$, \quad Weihua He$^{a}${\footnote{Corresponding author, E-mail:hwh12@gdut.edu.cn}, \quad Hong-Jian Lai$^{b,c}$, \quad Jianping Li$^{a}$}\\
\small {$a.$ School of Mathematics and Statistics,}\\
\small {Guangdong University of Technology, Guangzhou, China}\\
\small {$b.$ School of Mathematics and Systems Science,}\\
\small {Guangdong Polytechnic Normal University, Guangzhou, China}\\
\small {$c.$ Department of Mathematics,}\\
\small {West Virginia University, Morgantown, USA}}
\date{}

\begin{document}
\maketitle
\begin{abstract}
The (global) average order of dominating sets of a graph is the average number of vertices of its dominating sets. Analogously, the local average order of dominating sets is the average number of vertices of its dominating sets containing a fixed vertex. In this paper, we show that the local average order of dominating sets of a graph with $n$ vertices is at least $\frac{n+1}{2}$, with equality if and only if the degree of the fixed vertex is $n-1$. Furthermore, for a graph on $n$ vertices without isolated vertices, we show that $\frac{5n-1}{6}$ is an upper bound for the local average order of dominating sets. Additionally, we give a proof of an exact formula for the local average order of dominating sets when the degree of the fixed vertex is $n-2$, and determine an upper bound for the local average order of dominating sets when the fixed vertex is an $l$-stem ($l \geq 2$).

\end{abstract}

\section{Introduction}

Given a simple finite undirected graph $G$, let $V(G)$ and $E(G)$ be the vertex set and the edge set, respectively. For a vertex $v \in V(G)$, a vertex $u$ adjacent to $v$ in $G$ is called a \textit{neighbor} of $v$. The set of all neighbors of $v$ is called the \textit{open neighborhood} of $v$, denoted by $N_G(v)$. We also use $N_G[v]$ to denote the \textit{closed neighborhood} of $v$, that is, $N_G[v] = N_G(v) \cup \{v\}$. When there is no ambiguity, these notations are abbreviated as $N(v)$ and $N[v]$, respectively. The degree of $v$ is denoted by $\deg(v)$. In particular, a vertex $v$ is called an \textit{isolated vertex} if $N(v) = \emptyset$. A subset $S$ of $V(G)$ is a \textit{dominating set} of $G$ if $N[S] = V(G)$, which implies that every vertex in $V(G)-S$ is adjacent to at least one vertex in $S$. The {\em domination number} $\gamma(G)$ is the order of the smallest dominating set of $G$. We call a vertex of degree one a \textit{leaf} and a vertex adjacent to a leaf a \textit{stem}. For a stem $v$, let $L_G(v)$ denote the collection of all leaf neighbors of $v$ in $G$ and $L_G[v]=L_G(v)\cup \left\{v\right\}$. When only referring to a single graph, the subscript $G$ can be omitted and we denote them simply as $L(v)$ and $L[v]$. A stem with $l$ leaves is called an {\em $l$-stem}.

Considerable research has been devoted to average graph parameters. In 1977, Doyle and Graver\cite{MD_OYLE1977147} provided a formula for calculating the mean distance of trees, and demonstrated that the maximum mean distance in a connected graph with $n$ vertices is attained on a path with value $\frac{n+1}{3}$, while complete graphs achieve the minimum average distance with value $1$.
Subsequently, an increasing number of average graph parameters were introduced. In 1983, Jamison\cite{MST_JAMISON1983207} introduced the average order of a subtree of a graph. Many researchers carried out in-depth studies of it, focusing on its extremal problems \cite{MST_CAMBIE2021103388,MST_Mol2019, li2024} or the impact of edge modifications on it \cite{MST_JAMISON198470,MST_Cameron2021,MST_Luo2023, wang2024, cambie2024}.
In addition, various average graph parameters such as the average sizes of independent sets\cite{Indep_Davies2017,Indep_Andriantiana2020}, the average sizes of matchings \cite{Matching_2020}, the average connectivity\cite{AC_BEINEKE200231}, and the average eccentricity\cite{AE_DANKELMANN20191261,AE_HOROLDAGVA2021109,AE_ALOCHUKWU202398} have been successively proposed.

The average order of dominating sets was introduced by Beaton and Brown\cite{DS_Beaton2021}. Let $\mathcal{D}(G)$ be the collection of all dominating sets of $G$. The {\em average order of dominating sets} of $G$ is denoted by $avd(G)$, which is given by
\[
avd(G) = \frac{\sum_{S \in \mathcal{D}(G)} |S|}{|\mathcal{D}(G)|}.
\]
Let $ d_k(G) = \left| \left\{ S \in \mathcal{D}(G) : |S| = k \right\} \right| $.
The quantity $avd(G)$ can also be defined as
\[
avd(G) =\frac{\sum_{k=\gamma(G)}^{|V(G)|} kd_k(G)}{\sum_{k=\gamma(G)}^{|V(G)|} d_k(G)} .
\]
For brevity, let
$|\mathcal{D}(G)| = \sum_{k=\gamma(G)}^{|V(G)|}d_k(G)=\Gamma(G)$
and $\sum_{S \in \mathcal{D}(G)} |S| = \sum_{k=\gamma(G)}^{|V(G)|}kd_k(G)=\Gamma'(G)$. Thus we have
$$avd(G) = \frac{\Gamma'(G)}{\Gamma(G)}.$$

For all graphs with $n$ vertices, it is shown that the complete graph $K_n$ minimizes and the empty graph maximizes the average order of dominating sets. In addition, bounds have been developed for the average order of dominating sets for graphs without isolated vertices, where it was shown that $avd(G)\leq \frac{3n}{4}$ for every graph $G$ of order $n$ without isolated vertices \cite{DS_Beaton2021}. But this upper bound is not best possible. Erey\cite{DS_EREY2023113127} advanced this research by proving that for any forest without isolated vertices, $avd(G)$ has an upper bound of $\frac{2n}{3}$. This result significantly improved the upper bound in \cite{DS_Beaton2021} provided by Beaton and Brown, while also characterizing the maximal graphs with respect to the average order of dominating sets in forests without isolated vertices, such that every non-leaf vertex is an $l$-stem ($l\leq2$).
Recently, Beaton and Cameron\cite{DS_Brown2024} generalized Erey's conclusion on forests to arbitrary graphs. They showed that every graph $G$ of order $n$ without isolated vertices satisfies $avd(G)\leq \frac{2n}{3}$, proving that this upper bound is attained only by star-like graphs.

For a fixed vertex $v$, let $\mathcal{D}_v(G)$ denote the collection of all dominating sets of $G$ containing $v$. Therefore, we define the {\em local average order of dominating sets (containing $v$)} as
\[
avd_v(G) = \frac{\sum_{S_v \in \mathcal{D}_v(G)} |S_v|}{|\mathcal{D}_v(G)|}.
\]
Let $d_k(G, v) = \left| \left\{ S_v \in \mathcal{D}_v(G) : |S_v| = k \right\} \right| $.
The quantity $avd_v(G)$ can also be defined by
\[
avd_v(G) =\frac{\sum_{k=\gamma(G)}^{|V(G)|} kd_k(G,v)}{\sum_{k=\gamma(G)}^{|V(G)|} d_k(G,v)}.
\]
For brevity, let $|\mathcal{D}_v(G)| =\sum_{k=\gamma(G)}^{|V(G)|} d_k(G,v)= \Gamma_v(G)$ and $\sum_{S_v \in \mathcal{D}_v(G)} |S_v|
=\sum_{k=\gamma(G)}^{|V(G)|} kd_k(G,v) =\Gamma_v'(G)$. Thus we have $$avd_v(G) = \frac{\Gamma_v'(G)}{\Gamma_v(G)}.$$

The global average describes the sizes of dominating sets collectively, but does not distinguish the roles of individual vertices. Conditioning on the inclusion of a fixed vertex allows us to examine how its degree and neighborhood structure affect the average size. More precisely, if $S$ is chosen uniformly from $\mathcal{D}(G)$ and $X=|S|$, then $avd_v(G)=\mathbb{E}[X\mid v\in S]$. This local quantity also records information about the variation in the sizes of dominating sets. Indeed, writing $\mu=avd(G)$ and $p_v=\Gamma_v(G)/\Gamma(G)$, we have $\sum_v p_v=\mu$ and $\sum_v p_v avd_v(G)=\mathbb{E}[X^2]$. Consequently,
\[
\frac{\sum_{v\in V(G)}p_v avd_v(G)}{\sum_{v\in V(G)}p_v}
=\mu+\frac{\operatorname{Var}(X)}{\mu}.
\]
Thus, bounds on the local averages relate the position of individual vertices to the distribution of dominating set sizes.

According to this definition, we can derive the values of $avd_v(G)$ for specific vertices in certain classes of graphs. For example,
in the complete graph with $n$ vertices $K_n$, for any vertex $v$,  $$avd_v(K_n)=\frac{n+1}{2}.$$
In the empty graph with $n$ vertices $\overline{K_n}$, for any vertex $v$,
$$avd_v(\overline{K_n}) = n.$$
In the star graph with $n$ vertices $K_{1,n-1}$, there are two distinct types of vertices. When $\deg(v)=n-1$, $$avd_v(K_{1,n-1}) = \frac{n+1}{2}.$$
When $\deg(v)=1$, $$avd_v(K_{1,n-1}) = \frac{(n+2)2^{n-3} + n - 1}{2^{n-2} + 1}.$$

This paper is structured as follows. In Section 2, we establish the lower bound for the local average order of dominating sets, and further characterize the necessary and sufficient conditions to achieve this lower bound. In Section 3, we establish upper bounds for the local average order of dominating sets in graphs without isolated vertices. In Sections 4 and 5, we focus on the scenario in which the fixed vertex $v$ satisfies $\deg(v)=n-2$ and is an $l$-stem, respectively. The last section contains concluding remarks and a conjecture.

\section{Lower bound}

Beaton and Brown \cite{DS_Beaton2021} showed that $\frac{n2^{n-1}}{2^n - 1}$ is a lower bound for the global average order of dominating sets of graphs with $n$ vertices. It is of interest to study the lower bound of $avd_v(G)$ for arbitrary graphs. A natural question arises: given a graph $G$ of order $n$, can we determine the minimum value of $avd_v(G)$ among all vertices $v \in V(G)$?

Let $X$ be a nonempty finite set and $\mathcal{P}(X)$ its powerset. For any nonempty subset $\mathcal{A} \subseteq \mathcal{P}(X)$, we define the average order $av(\mathcal{A})$ of $\mathcal{A}$ to be $av(\mathcal{A}) = \frac{1}{|\mathcal{A}|} \sum_{A \in \mathcal{A}} |A|.$ The Boolean lattice $\mathcal{B}(X)$ on a finite set $X$ is the partial order $(\mathcal{P}(X), \subseteq)$. An ideal of this partial order is a subset $\mathcal{A}$ of $\mathcal{P}(X)$ that is {\em closed downwards}, that is, if $\sigma_1 \subseteq \sigma_2$ and $\sigma_2 \in \mathcal{A}$, then $\sigma_1 \in \mathcal{A}$, as well. Based on this, we know the following lemma.

\begin{lemma}\cite[Proposition 2.3]{DS_Beaton2021}\label{l1}
Let $X$ be a nonempty finite set with $n$ elements, and let $\mathcal{A}$ be a nonempty ideal of $\mathcal{B}(X)$. Let $\mathcal{A}_k = \{ A \in \mathcal{A} \colon |A| = k \}$ for $0\leq k\leq n$. Then for all $k \leq \frac{n}{2}$, we have
\[|\mathcal{A}_k| \geq |\mathcal{A}_{n-k}|.\]

Hence $av(\mathcal{A}) \leq \frac{n}{2}.$
\end{lemma}

Now we can derive the lower bound of the local average order of dominating sets.

\begin{theorem}\label{t22}
Let $G$ be a graph of order $n$. For every $v\in V(G)$, $avd_v(G) \geq \frac{n+1}{2}$ with equality if and only if $\deg(v)=n-1$.
\end{theorem}

\begin{proof}
The assertion is immediate when $n=1$, so suppose that $n\geq2$. Let $\mathcal{I}$ be the collection of sets $V(G)\setminus S$ with $S\in\mathcal{D}_v(G)$. Then $\mathcal{I}$ is a nonempty ideal of the Boolean lattice on $V(G)\setminus\{v\}$. By Lemma \ref{l1},
\[
avd_v(G)=n-av(\mathcal{I})\geq n-\frac{n-1}{2}=\frac{n+1}{2}.
\]
If $\deg(v)=n-1$, every subset containing $v$ is a dominating set, so $\mathcal{I}=\mathcal{P}(V(G)\setminus\{v\})$ and equality holds.

Conversely, suppose that $\deg(v)<n-1$. Then $\{v\}$ is not a dominating set, so $\mathcal{I}$ is a proper ideal. For each $u\ne v$, the map $T\mapsto T\setminus\{u\}$ injects the members of $\mathcal{I}$ containing $u$ into those omitting $u$. Moreover, choose an inclusion-maximal $T\in\mathcal{I}$ and a vertex $u\in (V(G)\setminus\{v\})\setminus T$. Since $T\cup\{u\}\notin\mathcal{I}$, this injection is not surjective for that $u$. Summing the resulting inequalities over all $u\ne v$ gives $av(\mathcal{I})<(n-1)/2$, and hence $avd_v(G)>(n+1)/2$.
\end{proof}

\section{Upper bound}

For a graph $G$ with $n$ vertices, we have the trivial bound $avd_v(G) \leq n$, with equality if and only if $G=\overline{K_n}$. According to \cite{DS_Brown2024}, we know that for a graph $G$ with $n$ vertices and $r$ isolated vertices, $avd(G) \leq \frac{2n + r}{3}$. If $v$ is an isolated vertex in $G$, all dominating sets of $G$ must contain $v$, thus $avd_v(G) = avd(G) \leq \frac{2n + r}{3}$. Now, consider the vertex $v$ of a connected graph $G$, what is the upper bound of $avd_v(G)$? Similarly, we explore the upper bound of $avd_v(G)$ by $\deg(v)$ and the minimum degree $\delta$.

For a fixed dominating set $S$ of graph $G$, if $S - v$ is no longer a dominating set, we say $v\in S$ is {\em critical} with respect to $S$. The set of all critical vertices with respect to $S\in \mathcal{D}(G)$ is denoted by $a(S)$, that is,
\[
a(S) = \{v \in S : S - v \notin \mathcal{D}(G)\}.
\]
On the other hand, for a fixed vertex $v$ in graph $G$, we consider the dominating set $S$ where $v$ is critical with respect to $S$, and let $a_v(G) = \{S \in \mathcal{D}(G) : S - v \notin \mathcal{D}(G)\}$. Obviously, when $v$ ranges over all vertices in $G$, we obtain the number of all critical vertices.

Let $\mathcal{D}_{-v}(G)$ denote the collection of dominating sets which do not contain vertex $v$. Just as $|\mathcal{D}_v(G)| = \Gamma_v(G)$, we  also denote $|\mathcal{D}_{-v}(G)|=\Gamma_{-v}(G)$. Clearly, $\Gamma(G) = \Gamma_v(G) + \Gamma_{-v}(G)$. Now, we establish the quantitative relationship among $|a_v(G)|$, $\Gamma_v(G)$ and $\Gamma_{-v}(G)$.

\begin{lemma}\cite[proof of Lemma 3.1]{DS_Beaton2021}\label{l31}
Let $G$ be a graph of order $n$ and $v$ be a vertex of $G$, then
\[
|a_v(G)| = \Gamma_v(G)-\Gamma_{-v}(G).
 \]
\end{lemma}

\begin{lemma}\label{l32}
Let $G$ be a graph of order $n$, then
\[
\sum_{S_v \in \mathcal{D}_v(G)} |a(S_v)| = 2\Gamma_v'(G) - (n - 1)\Gamma_v(G) - \Gamma(G).
\]
\end{lemma}

\begin{proof}First, consider the case where $v$ is critical with respect to $S$. By Lemma \ref{l31}, 
\[
|a_v(G)| = \Gamma_v(G)-\Gamma_{-v}(G) = \Gamma_v(G) - (\Gamma(G) - \Gamma_v(G)) = 2\Gamma_v(G) - \Gamma(G).
\]

Next, consider the case for all other vertices except $v$. Let $a_u^v(G) = \{ S_v \in \mathcal{D}_v(G) : S_v - u \notin \mathcal{D}_v(G) \}$ in which $u$ is critical with respect to $S_v$.
Similarly, let $\mathcal{D}_u^v(G)$ be the collection of all dominating sets in $\mathcal{D}_v(G)$ containing $u$ and $\mathcal{D}_{-u}^v(G)$ be those not containing $u$. Note that, when $u = v$, we simply write $\mathcal{D}_v^v(G)$ as $\mathcal{D}_v(G)$. By Lemma \ref{l31}, we also have $|a_u^v(G)| = |\mathcal{D}_u^v(G)| - |\mathcal{D}_{-u}^v(G)|.$
Since
\[
\sum_{u \in V(G)} |\mathcal{D}_u^v(G)| = \sum_{i=1}^{n} i \cdot d_i(G,v) = \Gamma_v'(G),
\]
it follows that
\[
\sum_{\substack{u \in V(G) \\ u \neq v}} |\mathcal{D}_u^v(G)| = \sum_{u \in V(G)} |\mathcal{D}_u^v(G)| - |\mathcal{D}_v(G)| = \Gamma_v'(G) - \Gamma_v(G).
\]
Also,
\begin{equation}
\sum_{\substack{u \in V(G) \\ u \neq v}} |\mathcal{D}_{-u}^v(G)| = \sum_{i=1}^{n} (n-i) \cdot d_i(G,v) = n\Gamma_v(G) - \Gamma_v'(G).
\label{eq:D_-V}
\end{equation}
Therefore
\begin{align*}
\sum_{S_v \in \mathcal{D}_v(G)} |a(S_v)| &= |a_v(G)| + \sum_{\substack{u \in V(G) \\ u \neq v}} |a_u^v(G)| \\
&= |\mathcal{D}_v(G)| - |\mathcal{D}_{-v}(G)| + \sum_{\substack{u \in V(G) \\ u \neq v}} \left( |\mathcal{D}_u^v(G)| - |\mathcal{D}_{-u}^v(G)| \right) \\
&= 2\Gamma_v(G) - \Gamma(G) + (\Gamma_v'(G) - \Gamma_v(G)) - (n\Gamma_v(G) - \Gamma_v'(G)) \\
&= 2\Gamma_v'(G) - (n - 1)\Gamma_v(G)- \Gamma(G).
\end{align*}
\end{proof}

In order to further explore the upper bound of $avd_v(G)$, we will make a more refined classification of $a(S)$ and $N(S)$. Here $N(S)=(\bigcup_{u\in S}N(u))\setminus S$. For $N(S)$, we define $N_1(S)$ as the set of vertices outside $S$ which have a single neighbor in $S$, and $N_2(S)$ as the set of vertices outside $S$ which have more than one neighbor in $S$. In other words,
\[
N_1(S) = \{v \in V - S : |N[v] \cap S| = 1\},
\]
\[
N_2(S) = \{v \in V - S: |N[v] \cap S| \geq 2\}.
\]

Since $S$ is a dominating set, it is clear that $N(S) = N_1(S) \cup N_2(S)$. Similarly, for $a(S)$, we define sets $a_1(S)$ and $a_2(S)$ such that $a(S) = a_1(S) \cup a_2(S)$ as follows:
\[
a_1(S) = \{v \in a(S) : N(v) \cap N_1(S) \neq \emptyset\},
\]
\[
a_2(S) = \{v \in a(S) : N(v) \cap N_1(S) = \emptyset\}.
\]

\begin{lemma}\cite[Lemma 3.2]{DS_Beaton2021}\label{l33}
Let $G$ be a graph. For any $S \in \mathcal{D}(G)$, $|a_1(S)| \leq |N_1(S)|$.
\end{lemma}

Let $P_v(G)$ be the collection of subsets of $V(G)\setminus N[v]$ that dominate all vertices of $G-v$. Equivalently, these are the dominating sets of $G-N[v]$ that also dominate $N(v)$ in $G$. Now, let $P_u^v(G)=\{S\in P_u(G):v\in S\}$. In particular, $P_v^v(G)=\emptyset$.

\begin{lemma}\label{l34}
Let $G$ be a graph. For any vertex $v \in V(G)$, we have $\frac{\Gamma(G)}{\Gamma_v(G)} < 2$.
\end{lemma}

\begin{proof}
Let $S=V(G)\setminus N(v)$. Then $S$ is a dominating set containing $v$: vertices outside $S$ are adjacent to $v$, and vertices in $S$ dominate themselves. However, $S\setminus\{v\}$ does not dominate $v$. Thus $S\in a_v(G)$, and Lemma \ref{l31} gives $\Gamma_v(G)-\Gamma_{-v}(G)=|a_v(G)|>0$. Consequently,
\[
\frac{\Gamma(G)}{\Gamma_v(G)}
=1+\frac{\Gamma_{-v}(G)}{\Gamma_v(G)}<2.
\]
\end{proof}

\begin{lemma}\cite[proof of Theorem 3.3]{DS_Beaton2021}\label{l35}
Let $G$ be a graph. Then we have $(2^{\deg(v)} - 1) |P_v(G)| \leq |\mathcal{D}_{-v}(G)|$ and
\[
\sum_{S \in \mathcal{D}(G)} |a_2(S)| = \sum_{v \in V(G)} |P_v(G)|.
\]
\end{lemma}

We are ready to improve the upper bound of $avd_v(G)$ based on the degree of the vertex $v$ and the minimum degree $\delta$.

\begin{theorem}\label{t36}
Let $G$ be a graph of order $n \geq 2$ with minimum degree $\delta\geq1$. For every vertex $v$ of $G$, we have
\[
avd_v(G) < \frac{2n - 1}{3} + \frac{2^{\deg(v)+1} - 1}{3(2^{\deg(v)} - 1)} + \frac{n - \deg(v) - 1}{3(2^{\delta+1} - 2)}.
\]
In particular, if $\deg(v)\geq \delta\geq1$, then
\[
avd_v(G) < \frac{5n + 2}{6}.
\]
\end{theorem}

\begin{proof}
Observe that no vertex in $a_2(S_v)$ is adjacent to any other vertex in $S_v$. Thus, for any $u \in N(v)$, it is known that $u \notin a_2(S_v)$, that is $P_u^v(G)=\emptyset$. For $u\notin N[v]$, adjoining $u$ injects $\mathcal{D}_{-u}^v(G)$ into $\mathcal{D}_u^v(G)$, so $|\mathcal{D}_{-u}^v(G)|\leq\Gamma_v(G)/2$. Also, the injection underlying Lemma \ref{l35} preserves the requirement that $v$ is selected: for $S\in P_u^v(G)$ and $\emptyset\ne T\subseteq N(u)$, the sets $S\cup T$ are distinct members of $\mathcal{D}_{-u}^v(G)$. Hence
$(2^{\deg(u)}-1)|P_u^v(G)|\leq|\mathcal{D}_{-u}^v(G)|$. Together with Lemma \ref{l35}, this yields
\begin{align}
\sum_{S_v \in \mathcal{D}_v(G)} |a_2(S_v)| &= |P_v(G)| + \sum_{u \notin N[v]} |P_u^v(G)| \label{eq:sum of a2(S)}\\
&\leq \frac{|\mathcal{D}_{-v}(G)|}{2^{\deg(v)} - 1} + \sum_{u \notin N[v]} \frac{|\mathcal{D}_{-u}^v(G)|}{2^{\deg(u)} - 1}\nonumber \\
&\leq \frac{\Gamma(G) - \Gamma_v(G)}{2^{\deg(v)} - 1} + \sum_{u \notin N[v]} \frac{1}{2^\delta - 1} \cdot \frac{\Gamma_v(G)}2 \nonumber\\
&= \frac{\Gamma(G) - \Gamma_v(G)}{2^{\deg(v)} - 1}+\frac{(n - \deg(v) - 1)\Gamma_v(G)}{2^{\delta+1} - 2}\nonumber.
\end{align}
By Lemma \ref{l33}, $|a_1(S_v)| \leq |N_1(S_v)|$. Together with Equation (\ref{eq:D_-V}) we obtain
$$\sum_{S_v \in \mathcal{D}_v(G)} |a_1(S_v)| \leq \sum_{S_v \in \mathcal{D}_v(G)} |N_1(S_v)| \leq \sum_{S_v \in \mathcal{D}_v(G)} |V(G) - S_v| = \sum_{\substack{u \in V(G) \\ u \neq v}} |\mathcal{D}^v_{-u}(G)| = n\Gamma_v(G) - \Gamma'_v(G).$$
By Lemma \ref{l32}, $\sum_{S_v \in \mathcal{D}_v(G)} |a(S_v)| = 2\Gamma_v'(G) - (n-1)\Gamma_v(G) - \Gamma(G)$, and hence from
\begin{equation}
\sum_{S_v \in \mathcal{D}_v(G)} |a(S_v)| = \sum_{S_v \in \mathcal{D}_v(G)} |a_1(S_v)| + \sum_{S_v \in \mathcal{D}_v(G)} |a_2(S_v)|,
\label{eq:sum of a(S_v)}
\end{equation}
we get
\[
2\Gamma_v'(G) - (n-1)\Gamma_v(G) - \Gamma(G) \leq n\Gamma_v(G) - \Gamma_v'(G) + \frac{(n - \deg(v) - 1)\Gamma_v(G)}{2^{\delta+1} - 2} + \frac{\Gamma(G) - \Gamma_v(G)}{2^{\deg(v)} - 1}.
\]
By reorganizing the above inequality, we have
\[
\frac{\Gamma_v'(G)}{\Gamma_v(G)} \leq \frac{2n - 1}{3} + \frac{2^{\deg(v)} \cdot \frac{\Gamma(G)}{\Gamma_v(G)} - 1}{3(2^{\deg(v)} - 1)} + \frac{n - \deg(v) - 1}{3(2^{\delta+1} - 2)}.
\]
By Lemma \ref{l34}, $\frac{\Gamma(G)}{\Gamma_v(G)} < 2$. Therefore,
\[
avd_v(G)=\frac{\Gamma_v'(G)}{\Gamma_v(G)} < \frac{2n-1}{3} + \frac{2^{\deg(v)+1}-1}{3(2^{\deg(v)}-1)} + \frac{n-\deg(v)-1}{3(2^{\delta+1}-2)}.
\]
In particular, when $\deg(v) \geq \delta \geq 1$,
\[
avd_v(G) < \frac{2n-1}{3} + \frac{2^{\deg(v)+1}-1}{3(2^{\deg(v)}-1)} + \frac{n-\deg(v)-1}{3(2^{\delta+1}-2)} \leq \frac{5n+2}{6}.
\]
\end{proof}

Theorem \ref{t36} gives an upper bound in terms of the degree of the fixed vertex and the minimum degree. To obtain a further bound that uses the degrees outside its closed neighborhood, we first consider a family of sets meeting prescribed subsets.

\begin{lemma}\label{l37}
Let $X$ be a set of $m$ elements and let $C_1,\ldots,C_r$ be nonempty subsets of $X$. Define
\[
\mathcal{F}=\{T\subseteq X:T\cap C_i\ne\emptyset\text{ for }1\leq i\leq r\}.
\]
If $T$ is chosen uniformly from $\mathcal{F}$, then
\[
\mathbb{E}|T|\leq\frac{m}{2}
+\frac12\sum_{i=1}^r\frac{|C_i|}{2^{|C_i|}-1}.
\]
For $r=0$, the family is $\mathcal{P}(X)$ and the sum is empty.
\end{lemma}

\begin{proof}
For $T\in\mathcal{F}$, let $c(T)$ be the number of elements $u\in T$ such that $T\setminus\{u\}\notin\mathcal{F}$. Since $\mathcal{F}$ is closed upwards, counting pairs of its members that differ in one element gives
\[
\sum_{T\in\mathcal{F}}(|T|-c(T))
=\sum_{T\in\mathcal{F}}(m-|T|),
\qquad
2\mathbb{E}|T|-m=\mathbb{E}c(T).
\]
Every element counted by $c(T)$ is the unique element of $T\cap C_i$ for some $i$. Distinct such elements require distinct indices, and hence
\[
c(T)\leq\sum_{i=1}^r\mathbf{1}_{\{|T\cap C_i|=1\}}.
\]
Fix an index $i$ and write $k=|C_i|$. Conditional on any feasible choice of $T\setminus C_i$, the allowed subsets of $C_i$ form a nonempty upward-closed family not containing the empty set. If this family contains no singleton, the conditional probability of $|T\cap C_i|=1$ is zero. Otherwise, let $q\geq1$ be its number of singletons. It contains every subset meeting the $q$ corresponding elements and therefore has at least $2^k-2^{k-q}$ members. Thus
\[
\Pr\bigl(|T\cap C_i|=1\mid T\setminus C_i\bigr)
\leq\frac{q}{2^k-2^{k-q}}\leq\frac{k}{2^k-1}.
\]
The last inequality follows because $q/(1-2^{-q})$ is increasing for positive integers $q$. Averaging over $T\setminus C_i$ and summing over $i$ proves the result.
\end{proof}

Now we can obtain another upper bound for $avd_v(G)$.

\begin{theorem}\label{t38}
Let $G$ be a graph of order $n$ and let $v_0\in V(G)$. Then
\[
avd_{v_0}(G)\leq\frac{n+1}{2}
+\frac12\sum_{u\notin N[v_0]}
\frac{\deg(u)+1}{2^{\deg(u)+1}-1}.
\]
In particular, if $G$ has no isolated vertices, then
\[
avd_{v_0}(G)\leq\frac{5n-2\deg(v_0)+1}{6}
\leq\frac{5n-1}{6}.
\]
\end{theorem}

\begin{proof}
Let $X=V(G)\setminus\{v_0\}$. A set $T\subseteq X$ satisfies $T\cup\{v_0\}\in\mathcal{D}_{v_0}(G)$ if and only if
\[
T\cap N[u]\ne\emptyset\qquad\text{for every }u\notin N[v_0].
\]
For such a vertex $u$, the set $N[u]$ is a subset of $X$ of size $\deg(u)+1$. Apply Lemma \ref{l37} to these closed neighborhoods. The bijection $T\mapsto T\cup\{v_0\}$ gives
\[
avd_{v_0}(G)=1+\mathbb{E}|T|
\leq\frac{n+1}{2}
+\frac12\sum_{u\notin N[v_0]}
\frac{\deg(u)+1}{2^{\deg(u)+1}-1}.
\]
If $G$ has no isolated vertices, then $\deg(u)\geq1$ and
$(\deg(u)+1)/(2^{\deg(u)+1}-1)\leq2/3$. Since there are $n-\deg(v_0)-1$ vertices outside $N[v_0]$, it follows that
\[
avd_{v_0}(G)\leq\frac{n+1}{2}
+\frac{n-\deg(v_0)-1}{3}
=\frac{5n-2\deg(v_0)+1}{6}
\leq\frac{5n-1}{6}.
\]
\end{proof}

\section{The local average order of dominating sets containing a fixed vertex with degree \texorpdfstring{$n-2$}{n-2}}

In fact, for some specific vertices in a graph $G$ with $n$ vertices, the local average order of dominating sets containing them can be calculated. Recall that we show the average order of dominating sets containing the central vertex or leaves in a star graph in Section 1. Therefore, in this section we focus on the local average order of dominating sets containing the fixed vertex with degree $n-2$.

Note that $ d_k(G) = \left| \left\{ S \in \mathcal{D}(G) : |S| = k \right\} \right| $. The domination polynomial of $G$ is defined as
\[
D(G, x) = \sum_{k=\gamma(G)}^{|V(G)|} d_k(G) x^k.
\]
We use the conventions $D(\emptyset,x)=1$ and $avd(\emptyset)=0$.

Since $avd(G)$ is equal to the logarithmic derivative of $D(G,x)$ evaluated at $x = 1$, the average order of dominating sets is also expressed as
\[
avd(G) = \frac{D'(G, 1)}{D(G, 1)} = \frac{\sum_{k=\gamma(G)}^{|V(G)|} k d_k(G)}{\sum_{k=\gamma(G)}^{|V(G)|} d_k(G)}.
\]

Now we turn to the local average order of dominating sets containing a fixed vertex.
The {\em local domination polynomial about the vertex $v$} is denoted by $D_{v\in S}(G,x)$, and defined by
\begin{equation}
D_{v\in S}(G, x) = \sum_{k=\gamma(G)}^{|V(G)|} d_k(G, v) x^k.\label{eq:D_vinS(G,x))}
\end{equation}
In the case without ambiguity, it is abbreviated as $D_v(G, x)$. Similarly, the local average order of dominating sets containing a fixed vertex $v$ is also expressed as
\[
avd_v(G) = \frac{D'_v(G, 1)}{D_v(G, 1)} = \frac{\sum_{k=\gamma(G)}^{|V(G)|} k d_k(G, v)}{\sum_{k=\gamma(G)}^{|V(G)|} d_k(G, v)}.
\]

%

Now we can determine the local average order of dominating sets containing a fixed vertex $v$, where $deg(v) =n-2$.

\begin{theorem}\label{t42}
Let $G$ be a connected graph of order $n\geq 3$, and let
$v\in V(G)$ satisfy $\deg(v)=n-2$. Suppose that $u$ is the unique
vertex in $V(G)\setminus\{v\}$ not adjacent to $v$, then
\[
avd_v(G) =\frac{(n+1)2^{n-2}- (n-\deg(u))2^{n-3- \deg(u)} }{2^{n-1} - 2^{n-2-\deg(u)}}.
\]
\end{theorem}

\begin{proof}
Since $v$ is adjacent to every vertex except $u$, a set containing $v$ dominates $G$ if and only if it meets $N[u]$. Moreover, $v\notin N[u]$ and $|N[u]|=\deg(u)+1$. Among the sets containing $v$, those disjoint from $N[u]$ have generating polynomial $x(1+x)^{n-\deg(u)-2}$. Therefore,
\begin{equation}\label{eq:near_universal_polynomial}
D_v(G,x)=x\left((1+x)^{n-1}-(1+x)^{n-\deg(u)-2}\right).
\end{equation}
Evaluating this polynomial and its derivative at $x=1$ yields
\begin{align*}
D_v(G,1)&=2^{n-1}-2^{n-\deg(u)-2},\\
D'_v(G,1)&=(n+1)2^{n-2}-(n-\deg(u))2^{n-\deg(u)-3}.
\end{align*}
Taking the ratio $D'_v(G,1)/D_v(G,1)$ proves the formula.
\end{proof}

\begin{corollary}\label{cor:near-universal-sharpness}
Let $G$ be a connected graph of order $n\geq 3$, and let
$v\in V(G)$ satisfy $\deg(v)=n-2$. Let $u$ be the unique
vertex in $V(G)\setminus\{v\}$ not adjacent to $v$. Then
\[
avd_v(G)
=
\frac{n+1}{2}
+
\frac{\deg(u)+1}{2\bigl(2^{\deg(u)+1}-1\bigr)}.
\]
Consequently, the degree-sum upper bound in
Theorem~\ref{t38} is attained whenever
$\deg(v)=n-2$.
\end{corollary}

\begin{proof}
Put $d=\deg(u)$. By Theorem~\ref{t42},
\begin{align*}
avd_v(G)
&=
\frac{(n+1)2^{d+1}-(n-d)}
     {2(2^{d+1}-1)}\\
&=
\frac{n+1}{2}
+
\frac{d+1}{2(2^{d+1}-1)}.
\end{align*}
Since $V(G)\setminus N[v]=\{u\}$, this expression
coincides with the degree-sum upper bound in
Theorem~\ref{t38}.
\end{proof}

\section{The local average order of dominating sets containing a fixed stem}

The following lemma provides an upper bound for the global average order of dominating sets in a graph without isolated vertices.

\begin{lemma}\cite[Theorem 2.10]{DS_Brown2024}\label{l51}
Let $G$ be a graph with $n \geq 2$ vertices. If $G$ has no isolated vertices, then
\[
avd(G) \leq \frac{2n}{3},
\]
with equality if and only if $G$ is star-like, where every vertex is either a leaf or an $l$-stem with $l\leq2$.
\end{lemma}

In the following result, we consider a type of stem where each of its adjacent vertices is either a leaf or a stem, and prove the upper bound of the local average order of dominating sets containing this kind of stem.

\begin{theorem}\label{t52}
Let $G$ be a graph of order $n \geq 2$ without isolated vertices. Let $v$ be a stem of $G$ with $L(v) = \{v_1, \cdots, v_l\}$ for some integer $l\geq1$. Suppose that each vertex adjacent to $v$ is either a leaf or a stem. Then,
\[
avd_v(G) \leq \frac{4n - l + 2}{6},
\]
with equality if and only if every connected component of $G-L[v]$ is star-like.
\end{theorem}

\begin{proof}
If $l=n-1$, Theorem \ref{t22} gives $avd_v(G)=(n+1)/2=(4n-l+2)/6$, and $G-L[v]$ is empty. Suppose $l<n-1$. The graph $G-L[v]$ has no isolated vertices: any such vertex would have been an additional leaf adjacent to $v$ in $G$. Hence $avd(G-L[v])\leq2(n-l-1)/3$ by Lemma \ref{l51}. Each non-leaf neighbor of $v$ has a leaf neighbor retained in $G-L[v]$, and this leaf forces that neighbor to be dominated within $G-L[v]$. Thus, for every $S\in\mathcal{D}_v(G)$, the set $S\setminus L[v]$ dominates $G-L[v]$. Conversely, any dominating set of $G-L[v]$, together with $v$ and an arbitrary subset of $L(v)$, dominates $G$. Therefore,
\[
d_k(G, v) = \sum_{i=0}^{l} \binom{l}{i} d_{k-1-i}(G - L[v]).
\]
By the domination polynomial (\ref{eq:D_vinS(G,x))}) and the binomial theorem, we obtain that
\begin{align*}
D_v(G, x) &= D(G - L[v], x) \sum_{i=0}^{l} \binom{l}{i} x^{i+1} \\
          &= x D(G - L[v], x) \sum_{i=0}^{l} \binom{l}{i} x^i \\
          &= x(x + 1)^l D(G - L[v], x).
\end{align*}
Differentiating, we get
\[
D'_v(G, x) = \left[(x+1)^l + lx(x+1)^{l-1}\right] D(G - L[v], x) + x(x+1)^l D'(G - L[v], x).
\]
Evaluating the polynomial at $x=1$ gives us
\begin{align*}
D_v(G, 1) &= 2^l D(G - L[v], 1), \\
D'_v(G, 1) &= (2^l + l 2^{l-1}) D(G - L[v], 1) + 2^l D'(G - L[v], 1).
\end{align*}
Now, we can obtain
\begin{align}
avd_v(G) &= \frac{D'_v(G, 1)}{D_v(G, 1)} \nonumber\\
&= \frac{(2^l + l 2^{l-1}) D(G - L[v], 1) + 2^l D'(G - L[v], 1)}{2^l D(G - L[v], 1)} \nonumber\\
&= \left(1 + \frac{l}{2}\right) + avd(G - L[v]) \nonumber\\
&\leq 1 + \frac{l}{2} + \frac{2}{3}(n - l - 1)\label{ineq:Th5.2} \\
&= \frac{4n - l + 2}{6}.\nonumber
\end{align}
Furthermore, one can verify that as $l\geq1$, $avd_v(G) \leq \frac{4n - l + 2}{6} \leq \frac{4n + 1}{6}$. Obviously, the inequality (\ref{ineq:Th5.2}) is an equality if and only if $avd(G - L[v]) = \frac{2}{3}(n - l - 1)$. In other words, every connected component of $G-L[v]$ is star-like.
\end{proof}

The following is the main result in this section.

\begin{theorem}\label{t57}
Let $G$ be a graph of order $n \geq 2$ without isolated vertices and $v$ be an $l$-stem of $G$. For $l \geq 2$,
\[
avd_v(G) < \frac{4n + 1}{6}.
\]
\end{theorem}

\begin{proof}
We first consider $l=2$. Let $w_1,w_2$ be the two leaf neighbors of $v$, and put
\[
H=G-\{v,w_1,w_2\},\qquad h=n-3,\qquad
U=N_G(v)\setminus\{w_1,w_2\}.
\]
The graph $H$ has no isolated vertices. Indeed, an isolated vertex of $H$ would have had $v$ as its only neighbor in $G$, contradicting the assumption that $v$ has exactly two leaf neighbors. Define
\[
\mathcal{A}=\{T\subseteq V(H):V(H)\setminus U\subseteq N_H[T]\},
\qquad P(x)=\sum_{T\in\mathcal{A}}x^{|T|}.
\]
When $v$ is selected, its two leaves can be chosen freely and the vertices in $U$ are already dominated. Therefore,
\begin{equation}\label{eq:two_leaf_local}
D_v(G,x)=x(1+x)^2P(x).
\end{equation}
Write
\[
A=P(1),\quad B=D(H,1),\quad
\alpha=\frac{P'(1)}{P(1)},\quad \beta=avd(H).
\]
Then $avd_v(G)=2+\alpha$. Also, $\mathcal{D}(H)\subseteq\mathcal{A}$, so $A\geq B>0$.

If $A=B$, the two families are equal and $P(x)=D(H,x)$. By Lemma \ref{l51}, with the empty-graph convention when necessary,
\[
avd_v(G)=2+\beta\leq2+\frac{2h}{3}
=\frac{2n}{3}<\frac{4n+1}{6}.
\]
Now suppose $A>B$. Choose a positive integer $r$ such that
\begin{equation}\label{eq:finite_copy_choice}
A^r>hB^r.
\end{equation}
Take $r$ disjoint copies $H_1,\ldots,H_r$ of $H$, add a vertex $z$ with two new leaf neighbors, and join $z$ to the copy of each vertex of $U$ in every $H_i$. Denote the resulting graph by $J_r$. It has no isolated vertices and has order $rh+3$. Separating its dominating sets according to whether they contain $z$ gives
\begin{equation}\label{eq:finite_copy_polynomial}
D(J_r,x)=x(1+x)^2P(x)^r+x^2D(H,x)^r.
\end{equation}
In the second term, both leaves must be selected and each copy of $H$ must be dominated internally. Taking the logarithmic derivative of \eqref{eq:finite_copy_polynomial} at $x=1$ yields
\[
avd(J_r)=2+r\frac{4A^r\alpha+B^r\beta}{4A^r+B^r}.
\]
Applying Lemma \ref{l51} to $J_r$ and rearranging, we obtain
\begin{equation}\label{eq:finite_copy_bound}
4A^r\left(\alpha-\frac{2h}{3}\right)
+B^r\left(\beta-\frac{2h}{3}\right)\leq0.
\end{equation}
Suppose, to the contrary, that $avd_v(G)\geq(4n+1)/6$. Since $n=h+3$, this implies $\alpha-2h/3\geq1/6$. As $\beta\geq0$, the left-hand side of \eqref{eq:finite_copy_bound} is at least
\[
\frac{4A^r}{6}-\frac{2hB^r}{3}
=\frac23(A^r-hB^r)>0,
\]
contrary to \eqref{eq:finite_copy_bound}. This proves the result for $l=2$.

Finally, let $l\geq2$, and obtain $G_0$ by deleting $l-2$ of the leaf neighbors of $v$. Then $G_0$ has no isolated vertices, has order $n-l+2$, and $v$ has exactly two leaf neighbors in $G_0$. Since the deleted leaves are free choices in a dominating set containing $v$,
\[
D_v(G,x)=(1+x)^{l-2}D_v(G_0,x).
\]
The case already proved gives
\begin{align*}
avd_v(G)&=\frac{l-2}{2}+avd_v(G_0)\\
&<\frac{l-2}{2}+\frac{4(n-l+2)+1}{6}
=\frac{4n-l+3}{6}\leq\frac{4n+1}{6}.
\end{align*}
\end{proof}

\section{Concluding remarks}

By Theorem \ref{t52}, $avd_v(G)\leq(4n+1)/6$ whenever $v$ is a stem and each of its neighbors is either a leaf or a stem. Theorem \ref{t57} establishes the strict inequality for every $l$-stem with $l\geq2$, without any further condition on its neighbors. The remaining case concerns $1$-stems with non-leaf neighbors that are not stems. This leads to the following conjecture.

\begin{conjecture}\label{conj61}
Let $G$ be a graph of order $n\geq2$ without isolated vertices, and let $v$ be an $l$-stem of $G$, where $l\geq1$. Then
\[
avd_v(G)\leq\frac{4n+1}{6},
\]
with equality if and only if $l=1$ and $G$ is star-like.
\end{conjecture}

\end{document}